\documentclass[12pt]{amsart}

\usepackage{amsfonts,amsmath,amsthm}
\usepackage{amssymb}
\usepackage{latexsym}
\usepackage[utf8]{inputenc}
\usepackage[english]{babel}
\usepackage{graphicx}
\usepackage{epstopdf}
\usepackage{tikz}
\usetikzlibrary{shapes,arrows,backgrounds}
\usepackage{subcaption}
\usepackage{hyperref}
\usepackage{array}
\usepackage{xypic}
\usepackage{ifsym}
\usepackage{url}
\usepackage{booktabs}

\newtheorem{theorem}{Theorem}[section]
\newtheorem*{theorem-non}{Theorem}
\newtheorem*{teorema-non}{Theorem}
\newtheorem{lemma}[theorem]{Lemma}
\newtheorem{teorema}[theorem]{Theorem}

\newtheorem*{proposition-non}{Proposition}
\newtheorem{prop}[theorem]{Proposition}

\newtheorem*{coro-non}{Corollary}

\newtheorem{definition}[theorem]{Definition}
\newtheorem{example}[theorem]{Example}

\newtheorem*{obs-non}{Observation}
\newtheorem*{conjecture-non}{Conjecture}
\theoremstyle{remark}
\newtheorem*{question-non}{Question}

\newtheorem{nota}[theorem]{}

\newtheorem{remark}[theorem]{Remark}

\def\cl{\overline}

\def\interior{\operatorname{int}}
\def\St{\operatorname{St}}

\def\R{\mathbb R}
\def\Q{\mathbb Q}
\def\Irr{\mathbb I}

\def\N{\mathbb N}
\def\B{\mathcal B}
\def\P{\mathcal P}
\def\AA{\mathcal A}
\def\T{\mathcal T}
\def\S{\mathcal S}

\def\RR{\mathcal R}
\def\L{\mathcal T}
\def\LL{\mathcal L}
\def\el2{\mathcal D}

\def\TRXY{\T:R(X)\to R(Y)}
\def\TBXY{\T:B_X\to B_Y}
\def\tauXY{\tau:X\to Y}
\def\tauX0Y{\tau:X_0\to Y_0}

\def\X0X{X_0\subset X}
\def\Y0Y{Y_0\subset Y}

\title{Order isomorphisms between bases of topologies}

\begin{document}

\author
{Javier~Cabello~Sánchez}
\address{Departamento de Matem\'{a}ticas, Universidad de Extremadura, 
Avda. de Elvas s/n, 06006 Badajoz. Spain}

\thanks{e-mail address: coco@unex.es}


\begin{abstract}
In this paper we will study the representations of isomorphisms 
between bases of topological spaces. It turns out that the perfect 
setting for this study is that of regular open subsets of complete metric 
spaces, but we have been able to show some results about arbitrary bases 
in complete metric spaces and also about regular open subsets in Hausdorff 
regular topological spaces. 
\end{abstract}

\maketitle

\section{Introduction}
Back in the 1930's, Stefan~Banach and Marshall~Stone proved one of the most 
celebrated results in Functional Analysis. The usual statement the reader 
can find of the Banach-Stone Theorem is, give or take, the following: 

\begin{theorem-non}
Let $X$ and $Y$ be compact Hausdorff spaces and let $T:C(X)\to C(Y)$ 
be a surjective linear isometry. Then there exist a homeomorphism 
$\tauXY$ and $g\in C(Y)$ such that $|g(y)|=1$ for all $y\in Y$ and
$(Tf)(y)=g(y)f(\tau(y))$ for all $y\in Y, f\in C(X)$.
\end{theorem-non}

The result is, however, much deeper. They were able to determine $X$ by means 
of the structure of $C(X)$, in the sense that $X$ turns out to be homeomorphic 
to the set of extreme points of the unit sphere of $(C(X))^*$ (after quotienting 
by the sign). 

Since then, similar results began to appear, as Gel'fand-Kolmogorov Theorem 
(\cite{GK}) or the subsequent works by Milgram, Kaplansky or Shirota, 
(\cite{Kap1, Kap2, Milgram, shirota}). In spite of this rapid development, after 
Shirota's 1952 work --which we will discuss later-- a standstill 
lasts until the last few years of the XXth century. Then the topic forks in 
two different ways. On the one hand, there are mathematicians that begin to suspect 
that the proof of~\cite[Theorem~6]{shirota} did not work, so they began 
to study lattice isomorphisms between spaces of uniformly continuous functions 
(see, e.g., \cite{GaJaU}). On the other hand, it begins to appear a significant 
amount of papers that deal with representation of isomorphisms between other 
spaces of functions or, in general, between subsets of $C(X)$ and $C(Y)$, 
see (\cite{FCPosit, GaJaH, GaJaExt, Grab, JVlLO, Mrcun}). 

Anyhow, the papers where we find some of the most accurate results about 
isomorphisms of spaces of uniformly continuous functions (\cite{CCshirota, JCSSharp}), 
Lipschitz functions (\cite{CClipschitz}) and smooth functions (\cite{CCsmooth}) 
have something in common: the result labelled in the present paper as 
Lemma~\ref{densos}; the interested reader should take a close look at~\cite{Leung}, 
where the authors were able to unify all these results and find new ones. This 
lemma has been key in these works, and has recently lead to similar results, 
see (\cite{JCSJAJA, JAJASmooth}). In the present paper we study Lemma~\ref{densos}, 
generalising it in two ways and providing a more accurate description of 
the isomorphisms between lattices of regular open sets in complete metric spaces. 

In the first part of the second section, we shall restrict ourselves to the 
study of complete metric spaces and order preserving bijections between arbitrary 
bases of their topologies. Namely, we will show that given a couple of complete 
metric spaces, say $X$ and $Y$, every order preserving bijection between bases 
of their topologies induces a homeomorphism between dense G$_\delta$ subspaces  
$X_0\subset X$ and $Y_0\subset Y$ --subspaces that can be endowed with (equivalent) 
metrics that turns them into complete spaces. Later, we restrict ourselves to the bases of 
regular open sets on the wider class of Hausdorff, regular topological spaces 
and show that whenever $X_0\subset X$ is dense, the lattices $R(X)$ and $R(X_0)$ 
of regular open subsets are naturally isomorphic and analyse some consequences 
of this. Joining both parts we get an explicit representation of every 
isomorphism between lattices of regular open sets in complete metric spaces that 
may be considered as the main result in this paper.

\begin{nota}
Apart from this Introduction, the present paper has Section~\ref{SMain}, where 
we prove the main results of the paper, and Section~\ref{skip}, that 
contains some remarks and applications of the main results. 
\end{nota}

\begin{nota}
In this paper, $X$ and $Y$ will always be topological spaces. 

We will denote the interior of $A\subset X$ as $\interior_X A$, unless the space 
$X$ is clear by the context, in which case we will just write $\interior A$. 
The same way, $\cl{A}^X$ or $\cl{A}$ will denote the closure of $A$ in $X$. 

We will denote by $R(X)$ the lattice of regular open subsets of $X$ and $\B_X$ 
will be a basis of the topology of $X$, please recall that an open subset $U$ 
of some topological space $X$ is regular if and only if $U=\interior \cl{U}$. 

We say that $\L:\B_X\to \B_Y$ is an isomorphism when it is a bijection that 
preserves inclusion, i.e., when $\L(U)\subset\L(V)$ is equivalent to $U\subset V$. 
\end{nota}

\section{The main result}\label{SMain}

In this Section we will prove our main result, Theorem~\ref{regulares}. 
Actually, it is just a consequence of Theorem~\ref{teoremav2} and 
Proposition~\ref{UX0}, but as both results are more general than 
Theorem~\ref{regulares} we have decided to separate them. We have split the proof 
in several intermediate minor results. 

\begin{lemma}\label{densos}
Let $(X,d_X)$ and $(Y,d_Y)$ be complete metric spaces or locally compact 
metric spaces and $\B_X, \B_Y$, bases of their topologies. Suppose there 
is an isomorphism $\T:\B_X\to \B_Y$. Then, there exist 
dense subspaces $X_0\subset X$ and $Y_0\subset Y$ and a homeomorphism 
$\tau:X_0\to Y_0$ such that $\tau(x)\in\T(U)$ if and only if $x\in U\cap X_0$. 
\end{lemma}

\begin{proof}
The proof is the same as in \cite[Lemma 2]{CCshirota}. 
\end{proof}

\begin{remark}
In the conditions of Lemma~\ref{densos}, we will denote 
$$\RR_X(x)=\bigcap_{x\in U\in\B_X}\T(U);\quad 
\RR_Y(y)=\bigcap_{y\in V\in\B_Y}\T^{-1}(V).$$ 
What the proof of \cite[Lemma 2]{CCshirota} shows is that the subset 
$X_0$ is dense in $X$, where $X_0$ consists of the points $x\in X$ for which 
there exists $y\in Y$ such that $\RR_X(x)=\{y\}$ and $\RR_Y(y)=\{x\}$. 
Once we have that $X_0$ is dense, it is clear that the map sending each 
$x\in X_0$ to the only point in $\RR_X(x)$ is a homeomorphism. 
\end{remark}

The following Theorem is just a translation of the {\em Théorème fondamental} 
in~\cite{Lavrentieff}. 

\begin{teorema}\label{Lavrentieff}
If there exists a bicontinuous, univocal and reciprocal correspondence 
between two given sets (inside an $m$-dimensional space), it is possible to determine 
another correspondence with the same nature between the points of two $G_\delta$ sets 
containing the given sets, the second correspondence agreeing with the first 
in the points of the two given sets. 
\end{teorema}

A more general statement of Lavrentieff's Theorem can be found 
in~\cite[Theorem 24.9]{willard}:

\begin{teorema}[Lavrentieff]\label{Lavrentieff2}
If $X$ and $Y$ are complete metric spaces and $h$ is a homeomorphism of $A\subset X$ 
onto $B\subset Y$, then $h$ can be extended to a homeomorphism $h^*$ of $A^*$ onto 
$B^*$, where $A^*$ and $B^*$ are $G_\delta$-sets in $X$ and $Y$, respectively, and 
$A\subset A^*\subset \cl{A}$, $B \subset B^*\subset \cl{B}$. 
\end{teorema}

As for the following Theorem, the author has been unable to find Alexandroff's 
work \cite{Alexandroff}, but Hausdorff references the result in~\cite{Hausdorff} as 
follows:

\begin{teorema}[Alexandroff, \cite{Alexandroff, Hausdorff}]\label{Alexandroff}
Every $G_\delta$ subset in a complete space is homeomorphic to a complete space. 
\end{teorema}

Combining Theorems~\ref{Lavrentieff2} and~\ref{Alexandroff} with Lemma~\ref{densos}
we obtain: 

\begin{prop}\label{casiteorema}
Let $X$ and $Y$ be complete metric spaces and $\TBXY$ an inclusion preserving bijection. 
Then, there exist a complete metric space $Z$ and dense $G_\delta$ subspaces 
$X_1\subset X,\ Y_1\subset Y$ such that $Z, X_1$ and $Y_1$ are mutually homeomorphic. 
\end{prop}


Of course, if $Z$ is an in Proposition~\ref{casiteorema} then every dense 
$G_\delta$ subset $Z'\subset Z$ fulfils the same, so it is clear that there is 
no {\em minimal} $Z$ whatsoever. In spite of this, it is very easy 
to determine some {\em maximal} $Z$: 

\begin{teorema}\label{teorema}
The greatest possible space $Z$ in the preceding Proposition is (homeomorphic 
to) $(X_0,d_Z)$, where $X_0$ is the subset given in Lemma~\ref{densos} and 
\begin{equation}\label{dZ}
d_Z(x,x')=\max\{d_X(x,x'), d_Y(\tau(x),\tau(x'))\}.
\end{equation}
\end{teorema}

A more explicit, though less clear, way to state Theorem~\ref{teorema} is the following: 

\begin{teorema}\label{teoremav2}
The metric $d_Z$ makes $X_0$ complete and, moreover, if $Z'$ embeds into both $X$ and 
$Y$ respectively via $\phi'_X$ and $\phi'_Y$ in such a way that $\phi'_X(z)\in U$ 
if and only if $\phi'_Y(z)\in \T(U)$, then $\phi'_X$ embeds $Z'$ into $X_0$. 
\end{teorema}

\begin{proof}
For the first part, take a $d_Z$-Cauchy sequence $(x_n)$ in $X_0$ and let 
$y_n=\tau(x_n)$ for every $n$. It is clear that both $(x_n)$ and $(y_n)$ are 
$d_X$-Cauchy and $d_Y$-Cauchy, respectively, so let $x=\lim(x_n)\in X, 
y=\lim(y_n)\in Y$, these limits exist because $X$ and $Y$ are complete. 
It is clear that any sequence $(\tilde{x}_n)\subset X_0$ converges to $x$ 
if and only if $y=\lim(\tau(\tilde{x}_n))$. This readily implies that 
$\RR_X(x)=\{y\}$, so $x\in X_0$ and this means $(X_0,d_Z)$ is complete. 

Now we must see that every metric space 
$Z'$ that embeds in both $X$ and $Y$ is embeddable in $X_0$, whenever 
the embeddings respect the isomorphism between the bases. For this, as 
$X_0$ is endowed with the restriction of the topology of $X$ and $Z'$ is homeomorphic 
to $\phi'_X(Z')\subset X$, the only we need is $\phi'_X(Z')\subset X_0$. 
So, suppose $x\in \phi'_X(Z')\setminus X_0$ and let $z\in(\phi'_X)^{-1}(x)$. 
As $Z'$ also embeds in $Y$, there exists $y=\phi'_Y(z)\in \phi'_Y(Z')\setminus Y_0$, 
too, with the property that $x\in U$ if and only if $y\in\T(U)$. 
By the very definition of $X_0$ and $Y_0$ this means that 
$x\in X_0, y\in Y_0$ and $\tau(x)=y$. 
\end{proof}

Now, we approach Proposition~\ref{UX0}, the main result about regular topological 
spaces. For this, the following elementary results will come in handy. 

\begin{lemma}\label{caracterizacion}
Let $Z$ be a topological space and $A\subset Z$. $A$ is a regular open subset of $Z$ 
if and only if for every open $V\subset Z,\ V\subset \cl{A}$ implies $V\subset A$.
\end{lemma}

\begin{proof}
It is obvious. 
\end{proof}

\begin{lemma}\label{DensoAbierto}
Let $X$ be a topological space. Whenever $Y\subset X$ is dense and $U\subset X$ is open, one has 
$\cl{U}^X=\cl{U\cap Y}^X.$ 
\end{lemma}

\begin{proof}
Let $x\in \cl{U}^X$. This is equivalent to the fact that every open neighbourhood $V$ of 
$x$ has nonempty intersection with $U$. So, $V\cap U$ is a nonempty open subset of $X$ and 
the density of $Y$ implies that $V\cap U\cap Y$ is also nonempty, so $x\in \cl{U\cap Y}^X$ 
and we have $\cl{U}^X\subset\cl{U\cap Y}^X$. The other inclusion is trivial. 
\end{proof}

\begin{lemma}\label{UVW}
Let $X$ be a topological space and $U, V\in R(X)$ such that $U\not\subset V$. Then, 
there is $\emptyset\neq W\in R(X)$ such that $W\cap U=\emptyset$ and $W\subset V$. 
\end{lemma}

\begin{proof}
Actually, $U\setminus\cl{V}$ is regular because $U$ and $X\setminus\cl{V}$ are 
regular and $\left(U\setminus\cl{V}\right)=U\cap \left(X\setminus\cl{V}\right)$. This 
set is nonempty because $U\subset\cl{V}$, along with the monotonicity of the 
interior operator, would imply 
$$U=\interior_X U\subset\interior_X\left( \cl{V}\right)=V. $$
\end{proof}

\begin{remark}\label{UVWH}
If in Lemma~\ref{UVW} $X$ is regular and Hausdorff, then 
$V$ can be taken as any open subset that contains $U$ strictly. 
\end{remark}

\begin{prop}\label{UX0}
Let $X$ be a topological space and $Y\subset X$ a dense subset. Then $\TRXY$, defined as 
$\T(U)=U\cap Y$, is a lattice isomorphism with inverse $\S(V)=\interior\cl{V}$. 
\end{prop}

\begin{proof}
We need to show that $\T$ and $\S$ are mutually inverse. 

Let $U\in R(X)$, the first we need to show is that $\T$ is well-defined, i.e., that 
$\T(U)=U\cap Y$ is regular in $Y$. 

Let $V\subset X$ an open subset such that $V\cap Y\subset 
\cl{U\cap Y}^Y$. Then, as the closure in $X$ preserves inclusions, we have 
$$\cl{V}^X=\cl{V\cap Y}^X\subset \cl{\cl{U\cap Y}^{Y}}^X\subset \cl{U}^X,$$ 
where the first equality holds because of Lemma~\ref{DensoAbierto}. Taking interiors 
in $X$ also preserves inclusions, so we obtain 
$$V\subset\interior_X\left(\cl{V}^X\right)\subset \interior_X\left(\cl{U}^X\right)=U,$$ 
which readily implies that $V\cap Y\subset U\cap Y$ 
and we obtain $V\cap Y\in R(Y)$ from Lemma~\ref{caracterizacion}. 
It is clear that $\S(V)\in R(X)$ for every $V\in R(Y)$, so both maps are well-defined. 

Furthermore, Lemma~\ref{DensoAbierto} implies that, for any regular $U\subset X,$ 
$$\S\circ\T(U)=\S(U\cap Y)=\interior_X\left(\cl{U\cap Y}^X\right)=\interior_X\left(\cl{U}^X\right)=U. $$
As for the composition $\T\circ\S$, we have 
$$\T\circ\S(V)=\T\left(\interior_X\left(\cl{V}^X\right)\right)=\interior_X\left(\cl{V}^X\right)\cap Y$$
for any $V\in R(Y)$. Let $W\subset X$ be an open subset for which $V=W\cap Y$, 
the very definition of inherited topology implies that there exists such $W$. 
The previous equalities can be rewitten as 
$$\T\circ\S(W\cap Y)=
\T\left(\interior_X\left(\cl{W\cap Y}^X\right)\right)=
\T\left(\interior_X\left(\cl{W}^X\right)\right)=
\interior_X\left(\cl{W}^X\right)\cap Y,$$
so we need $W\cap Y=\interior_X\left(\cl{W}^X\right)\cap Y$. 
It is clear that $W\cap Y\subseteq\interior_X\left(\cl{W}^X\right)\cap Y$, 
so what we need is $\interior_X\left(\cl{W}^X\right)\cap Y\subseteq W\cap Y$. 
Both subsets are regular in $Y$, so if this inclusion does not hold, 
there would exist an open $H\subset Y$ 
$$H\neq\emptyset,\quad H\subset \interior_X\left(\cl{W}^X\right)\cap Y,
\quad H\cap(W\cap Y)=\emptyset,$$
so, by Lemma~\ref{UVW} there is an open $G\subset X$ such that $H=G\cap Y$ and so 
\begin{equation}\label{GYW}
G\cap Y\neq\emptyset,\quad G\cap Y\stackrel{(*)}\subset \interior_X\left(\cl{W}^X\right)\cap Y,
\quad (G\cap Y)\cap(W\cap Y)=\emptyset,
\end{equation}
and this is absurd. Indeed, the inclusion marked with $(*)$ implies that 
we may substitute $G$ by $G\cap\interior_X\left(\cl{W}^X\right)$, so both 
equalities in (\ref{GYW}) hold for some open $G\subset \interior_X\left(\cl{W}^X\right)$. 
As $Y$ is dense and $G$ and $W$ are open, the last 
equality implies that $G\cap W=\emptyset$. Of course, this implies 
$G\cap \interior_X\left(\cl{W}^X\right)=\emptyset$, which means $G=\emptyset$ 
and we are done. 
\end{proof}

Now we are in conditions to state our main result: 
\begin{theorem}\label{regulares} 
Let $X$, $Y$ and $Z$ be complete metric spaces, $\phi_X:Z\hookrightarrow X$ 
and $\phi_Y:Z\hookrightarrow Y$ be continuous, dense, embeddings and 
$X_0=\phi_X(Z).$  
Then, $\TRXY$ given by the composition 
$$U\mapsto U\cap X_0\mapsto \phi_X^{-1}(U\cap X_0)\mapsto \phi_Y(\phi_X^{-1}(U\cap X_0))
\mapsto \interior\left(\cl{\phi_Y(\phi_X^{-1}(U\cap X_0))}\right)$$
is an isomorphism between the lattices of open regular subsets of $X$ and $Y$ 
and every isomorphism arises this way. 
\end{theorem}

The ``every isomorphism arises this way" part is due to Theorem~\ref{teoremav2}, 
while the ``the composition is an isomorphism" part is consequence of 
Proposition~\ref{UX0}. 

\section{Applications and remarks}\label{skip}
In this Section, we are going to show how Proposition~\ref{UX0} gives in an 
easy way some properties of $\beta\N$ and conclude with a couple of examples 
that show that the hypotheses imposed in the main results are necessary. 
But first, we need to deal with an error in some outstanding work. 
In~\cite{CClipschitz} F.~Cabello and the author of the present paper 
showed that some results in~\cite{shirota} were not properly proved. 
Later, in~\cite{CCshirota} the same authors proved that, even when the 
proof of \cite[Theorem 6]{shirota} was incorrect, the result was true. 
Now, we are going to explain what the error was. 
The following Definitions and Theorems can be found in~\cite{shirota}: 

\begin{definition}[Definition 2]\label{RL}
A distributive lattice with smallest element 0 satisfying Wallman's disjunction 
property is an {\em R-lattice} if there exists a binary relation $\gg$ in $L$ 
which satisfies: 
\begin{itemize} 
\item If $h\geq f$ and $f\gg g$, then $h\gg g$. 
\item If $f_1\gg g_1$ and $f_2\gg g_2$, then $f_1\wedge f_2\gg g_1\wedge g_2$. 
\item If $f\gg g$, then there exists $h$ such that $f\gg h\gg g$. 
\item For every $f\neq 0$ there exist $g_1$ and $g_2\neq 0$ such that 
$g_1\gg f\gg g_2$. 
\item If $g_1\gg f\gg g_2$, then there exists $h$ such that $h\vee f=g_1$ and 
$h\wedge g_2=0$. 
\end{itemize} 
\end{definition}

Immediately after Definition~2 we find this:

\begin{theorem}[Theorem 1]
A distributive lattice with smallest element 0 is an R-lattice if and only if 
it is isomorphic to a sublattice of the lattice of all regular open sets on a 
locally compact space $X$. This sublattice is an open base and its elements 
have compact closures. 
\end{theorem}

The open regular set in $X$ associated to $f\in L$ is denoted by $U(f)$. With 
this notation, the next statement is: 

\begin{theorem}[Theorem 2]\label{sh2}
Let $L$ be an R-lattice. Then there exists uniquely a locally compact space $X$ 
which satisfies the property in Theorem 1 and where $U(f)\supset \cl{U(g)}$ 
if and only if $f\gg g$. 
\end{theorem}

Our Proposition~\ref{UX0} contradicts the 
uniqueness of $X$ in the statement of Theorem~2 and we may actually explicit a 
lattice isomorphism between $R(X)$ and $R(Y)$ for different compact metric 
spaces $X$ and $Y$. Namely, we just need to take the simplest compactifications 
of $\R$ and the composition of the lattice isomorphisms predicted by 
Proposition~\ref{UX0}: 

\begin{example}
Let $X=\R\cup\{-\infty,\infty\}$ and $Y=\R\cup\{N\}$. Then, 
$\TRXY$, defined by 
$$\T(U)=\left\{
\begin{array}{l l}
U & \mathrm{\ if \ } U\cap\{-\infty,\infty\}=\emptyset \\
U\cap\R & \mathrm{\ if \ } U\cap\{-\infty,\infty\}=\{\infty\}\\
U\cap\R & \mathrm{\ if \ } U\cap\{-\infty,\infty\}=\{-\infty\}\\
(U\cap\R)\cup\{N\} & \mathrm{\ if \ } \{-\infty,\infty\}\subset U
\end{array}
\right.$$
is a lattice isomorphism whose inverse is given by 
$$\S(V)=\left\{
\begin{array}{l l}
V & \mathrm{\ if \ } N\not\in V\\
(V\cap\R)\cup\{-\infty,\infty\} & \mathrm{\ if \ } N\in V
\end{array}
\right.$$
\end{example}
 
It seems that the problem here is that the definition of R-lattice, Definition~2, 
does not include the relation $\gg$, but in Theorem~2 and its consequences the 
author considers $\gg$ as a unique, fixed, relation given by $(\LL,\leq)$. It is 
clear that the above spaces generate, say, different $\gg_X$ and $\gg_Y$ in the 
isomorphic lattices $R(X)$ and $R(Y)$. This leads to the error already noted 
in~\cite{CClipschitz}, Section~5. 

Actually, with the definition of R-lattice given in~\cite{shirota}, in seems that 
the original purpose of the definition is lost. Indeed, the relation $\gg$ 
may be taken as $\geq$ in quite a few lattices. This leads to a topology where 
every regular open set is clopen, in Section~\ref{betaN} we will see an example 
of a far from trivial topological space where this is true. 
Given a lattice $(\LL,\geq)$, the relation between each possible 
$\gg$ and the unique locally compact topological space given by Theorem~2 probably 
deserves a closer look. 

Anyway, if we include $\gg$ in the definition, then \cite[Theorem~2]{shirota} 
is true. So let us put everything in order. 

\begin{definition}[Shirota]\label{RL2}
Let $(\LL,\leq)$ be a distributive lattice with $\min\LL=0$ and $\gg$ be a 
relation in $\LL$. 
The triple $(\LL,\leq,\gg)$ is an R-lattice if the following holds: 
\begin{enumerate}
\item For every $a\neq b\in\LL$, there exists $h\in\LL$ such that 
either $a\wedge h=0$ and $b\wedge h\neq 0$ or the other way round. 
(Una forma de Wallman's disjunction Property). 
\item If $h\geq f$ and $f\gg g$, then $h\gg g$. 
\item If $f_1\gg g_1$ and $f_2\gg g_2$, then $f_1\wedge f_2\gg g_1\wedge g_2$. 
\item If $f\gg g$, then there exists $h$ such that $f\gg h\gg g$. 
\item For every $f\neq 0$ there exist $g_1$ and $g_2\neq 0$ such that 
$g_1\gg f\gg g_2$. 
\item If $g_1\gg f\gg g_2$, then there exists $h$ such that $h\vee f=g_1$ and 
$h\wedge g_2=0$. 
\end{enumerate}
\end{definition}

With Definition~\ref{RL2} everything works and this result remains valid, 
but it does not lead to the consequences stated there as Theorems 3 to 6. 

\begin{theorem}[Theorem 2]\label{sh2bien}
Let $(\LL,\leq,\gg)$ be an R-lattice. Then there exists uniquely a locally compact space $X$ 
which satisfies the property in Theorem 1 and where $U(f)\supset \cl{U(g)}$ 
if and only if $f\gg g$. 
\end{theorem}

\subsection{The Stone-\v Cech compactification of $\N$}\label{betaN}

We will analyse the isomorphism given in Proposition~\ref{UX0} when $Y=\N$ and 
$X=\beta\N$, the Stone-\v{C}ech compactification of $\N$. This is not going to 
lead to new results, but it seems to be interesting in spite of this. 
These are very different spaces, so it could be surprising the fact that 
they share the same lattice of regular open subsets. In any case, as $\N$ 
is discrete, every $V\subset\N$ is regular and this, along with 
Proposition~\ref{UX0}, implies that 
$$R(\beta\N)=\{\interior(\cl{U}):U\subset\N\}.$$
As $\beta\N$ is regular, every open $W\subset\beta\N$ is the union of its regular 
open subsets, and these regular open subsets are determined by the integers they 
contain, so $W$ is determined by a collection $\mathcal A_W$ of subsets of $\N$. 
Of course, if $W$ contains $\S(J)$ for some $J\subset\N$ and $I\subset J$ then 
$\S(I)\subset W$, too. This means that $\mathcal A_W$ is closed for inclusions. 
Furthermore, as $\cl{J}$ is open in $\beta\N$, 
see~\cite[p. 144, Subsection 3.9]{walker}, 
if $\S(I)\subset W$ and $\S(J)\subset W$ then $\S(I)\cup\S(J)=\S(I\cup J)$, 
so $\S(I\cup J)\subset W$ and $\mathcal A_W$ is closed for pairwise supremum. 
Summing up, $\mathcal A_W$ is an ideal of the lattice $\mathcal P(\N)$ 
for every proper open subset $\emptyset\subsetneq W\subsetneq\beta\N$
--and every $\S(J)\in R(\beta\N)$ is closed, so ``clopen" and ``regular open" 
are equivalent in $\beta\N$. 

It is also clear that every ideal $\mathcal A$ in $\P(\N)$ defines an open 
$W_\AA\in\beta\N$ as $\cup\{\S(I):{I\in\AA}\},$ that these identifications 
are mutually inverse and that $W\subset V$ if and only if $\AA_W\subset \AA_V$, so 
each maximal ideal
in $\P(\N)$ defines a maximal open proper subset of $\beta\N$. As $\beta\N$
is Hausdorff, these maximal open subsets are exactly $\beta\N\setminus\{x\}$
for some $x$, so each point is dually defined by a maximal ideal. In other
words, every point in $\beta\N$ is associated to an ultrafilter in $\P(\N)$.

As our final comment in this {\em Just for fun} Remark, we have that 
$\beta\N$ is the only Hausdorff compactification of $\N$ that fulfils:
\begin{itemize}
\item[$\spadesuit$] If $J, I\subset\N$ are disjoint, then their closures
in the compactification are disjoint, too, 
\end{itemize}
\noindent although this is just a particular case of a result by \v{C}ech, 
see~\cite[p. 25-26]{walker}. 

\subsection{The hypotheses are minimal}\label{sbsback}

In some sense, Theorem~\ref{regulares} is optimal. Here we see that there is no 
way to generalise it if we omit any of the hypotheses. 

\begin{remark}\label{concof}
Consider any infinite set $Z$ endowed with the cofinite topology $\tau_{cof}$. 
It is clear that every pair of nonempty open subsets of $Z$ meet, so 
every nonempty open subset is dense in $Z$ and this implies that the only regular 
open subsets of $Z$ are $Z$ and $\emptyset.$ Of course the same applies to any 
uncountable set endowed with the cocountable topology $\tau_{con}$, so 
$(\R,\tau_{cof})$ and $(\R,\tau_{con})$ have the same regular open subsets. 
Nevertheless, there is no way to identify homeomorphically any couple of dense 
subsets of $\R$ with each topology. In order to avoid this pathological behaviour 
we needed to consider only regular Hausdorff spaces since these spaces are the only 
reasonable spaces for which the regular subsets comprise a base of the topology. 
In other words, Theorem~\ref{regulares} will not extend to general topological 
spaces. 
\end{remark}

\begin{remark}\label{Stone}
Consider $X=[0,1]$ endowed with its usual topology and let $Y$ be its Gleason 
cover. The lattices $R(X)$ and $R(Y)$ are canonically isomprphic, but it is 
well known that no point in $Y$ has a countable basis of neighbourhoods, so 
$$\bigcap_{x\in U\in\B_X}\T(U)=\bigcap_{n=1}^\infty\T(B(x,1/n))$$
is never a singleton. 

It is remarkable that~\cite[Example~1.7.16]{GleasonStoneBoole} is the {\em only} 
place where the author has been able to find a statement that explicitly confirms
that the Gleason cover of some compact space $K$ is the same topological space 
as the Stone space associated to $R(K)$, i.e., $G_K=\St(R(K))$.
\end{remark}

\begin{remark}\label{SCQI}
Consider $A=\Q\cap[0,1]$, $B=\Irr\cap[0,1]$ and their Stone–Čech compactifications 
$X=\beta A$, $Y=\beta B$. 
There is a lattice isomorphism between $R(X)$ and $R(Y)$, say $\T$, given by 
the composition of the following isomorphisms: 
\begin{eqnarray}
& U\in R(X)\mapsto  U\cap A\in R(A), 
& U\in R(A)\mapsto  \interior\overline{U}\in R([0,1]), \nonumber \\
& U\in R([0,1])\mapsto   U\cap B \in R(B), 
& U\in R(B)\mapsto  \interior\overline{U}\in R(Y). \nonumber
\end{eqnarray}
In spite of this, it is intuitively evident that 
$$\RR_X(x)=\bigcap_{x\in U\in\B_X}\T(U)=\bigcap_{n=1}^\infty\T\big(\interior\big(\,\cl{B(x,1/n)}\,\big)\big)$$ 
is never a singleton for $x\in A$, and 
$$\RR_Y(y)=\bigcap_{y\in V\in\B_Y}\T^{-1}(V)=\bigcap_{n=1}^\infty\T^{-1}\big(\interior\big(\,\cl{B(y,1/n)}\,\big)\big)$$
is neither a singleton when $y\in B$. 
It seems clear that $X=A\sqcup\{\RR_Y(y):y\in B\}$ and $Y=B\sqcup\{\RR_X(x):x\in A\}$, 
so there is no point in $X_0$. 
\end{remark}

\begin{remark}\label{QI}
There is no ``non-complete metric spaces" result. Indeed, $\Irr$ and $\Q$ have 
the obvious isomorphism between their bases of regular open subsets and they are, 
nevertheless, disjoint subsets in $\R$. This means that when trying to generalise
Theorem~\ref{regulares} the problem may come not only from the lack of separation 
of the topologies as in Remark~\ref{concof}, from the excess or points as in 
Remark~\ref{Stone} or from the points in $X$ not squaring with those in $Y$ as 
in Remark~\ref{SCQI} but also from the, so to say, lack of 
points in the spaces even when they are metric. 
\end{remark} 

\section{Acknowledgements}
Supported in part DGICYT project PID2019-103961GB-C21. 

It is a pleasure to thank Professor Denny~H.~Leung for his interest in the present 
paper and for pointing out a mistake. 

\bibliographystyle{abbrv}

\addcontentsline{toc}{chapter}{Bibliography}

\bibliography{BaseCompleteSpaces}

\begin{thebibliography}{10}

\bibitem{Alexandroff}
P.~S. Alexandroff.
\newblock Sur les ensembles de la première classe et les ensembles abstraits.
\newblock {\em Compt. Rendus Acad. Sci Paris}, 178:185--187, 1924.

\bibitem{FCPosit}
F.~Cabello~S{\'a}nchez.
\newblock Homomorphisms on lattices of continuous functions.
\newblock {\em Positivity}, 12(2):341--362, 2008.

\bibitem{CCsmooth}
F.~Cabello~S{\'a}nchez and J.~Cabello~S{\'a}nchez.
\newblock Some preserver problems on algebras of smooth functions.
\newblock {\em Arkiv f{\"o}r Matematik}, 48(2):289--300, 2010.

\bibitem{CClipschitz}
F.~Cabello~S{\'a}nchez and J.~Cabello~S{\'a}nchez.
\newblock Nonlinear isomorphisms of lattices of {L}ipschitz functions.
\newblock {\em Houston J. Math}, 37(1):181--202, 2011.

\bibitem{CCshirota}
F.~Cabello~S{\'a}nchez and J.~Cabello~S{\'a}nchez.
\newblock Lattices of uniformly continuous functions.
\newblock {\em Topology and its Applications}, 160(1):50 -- 55, 2013.

\bibitem{JCSSharp}
J.~Cabello~S{\'a}nchez.
\newblock A sharp representation of multiplicative isomorphisms of uniformly
  continuous functions.
\newblock {\em Topology and its Applications}, 197:1--9, 2016.

\bibitem{JCSJAJA}
J.~Cabello~S\'anchez and J.~A. Jaramillo~Aguado.
\newblock A functional representation of almost isometries.
\newblock {\em Journal of Mathematical Analysis and Applications},
  445(2):1243--1257, 2017.

\bibitem{GleasonStoneBoole}
H.~G. Dales, F.~K. Dashiell, A.~T.-M. Lau, and D.~Strauss.
\newblock {\em Banach spaces of continuous functions as dual spaces}.
\newblock Springer, 2016.

\bibitem{JAJASmooth}
A.~Daniilidis, J.~A. Jaramillo, and F.~{Venegas}.
\newblock Smooth semi-{L}ipschitz functions and almost isometries between
  {F}insler manifolds.
\newblock {\em Journal of Functional Analysis}, 279(8):108662, 2020.

\bibitem{GaJaU}
M.~I. Garrido and J.~A. Jaramillo.
\newblock A {B}anach-{S}tone {T}heorem for uniformly continuous functions.
\newblock {\em Monatshefte für Mathematik}, 131:189--191, 2000.

\bibitem{GaJaExt}
M.~I. Garrido and J.~A. Jaramillo.
\newblock Variations on the {B}anach-{S}tone theorem.
\newblock {\em Extracta Mathematicae}, 17(3):351--383, 2002.

\bibitem{GaJaH}
M.~I. Garrido and J.~A. Jaramillo.
\newblock Homomorphisms on function lattices.
\newblock {\em Monatshefte f{\"u}r mathematik}, 141(2):127--146, 2004.

\bibitem{GK}
I.~Gel'fand and A.~M. Kolmogorov.
\newblock On rings of continuous functions on topological spaces.
\newblock {\em Dokl. Akad. Nauk. SSSR}, 22(1):11--15, 1939.

\bibitem{Grab}
J.~Grabowski.
\newblock Isomorphisms of algebras of smooth functions revisited.
\newblock {\em Archiv der Mathematik}, 85(2):190--196, 2005.

\bibitem{Hausdorff}
F.~Hausdorff.
\newblock Die {M}engen {G}$_\delta$ in vollst{\"a}ndigen {R}{\"a}umen.
\newblock {\em Fundamenta Mathematicae}, 6(1):146--148, 1924.

\bibitem{JVlLO}
A.~Jim{\'e}nez-Vargas and M.~Villegas-Vallecillos.
\newblock Order isomorphisms of little lipschitz algebras.
\newblock {\em Houston J. Math}, 34:1185--1195, 2008.

\bibitem{Kap1}
I.~Kaplansky.
\newblock Lattices of continuous functions.
\newblock {\em Bull. Amer. Math. Soc.}, 53:617--623, 1947.

\bibitem{Kap2}
I.~Kaplansky.
\newblock Lattices of continuous functions {II}.
\newblock {\em Amer. J. Math.}, 70:626--634, 1948.

\bibitem{Lavrentieff}
M.~M. Lavrentieff.
\newblock Contribution à la théorie des ensembles homéomorphes.
\newblock {\em Fund. Math.}, 6:149--160, 1924.

\bibitem{Leung}
D.~H. Leung and W.-K. Tang.
\newblock Nonlinear order isomorphisms on function spaces.
\newblock {\em Dissertationes Mathematicae}, 517:1--75, 2016.

\bibitem{Milgram}
A.~N. Milgram.
\newblock Multiplicative semigroups of continuous functions.
\newblock {\em Duke Math. J.}, 16:377--383, 1949.

\bibitem{Mrcun}
J.~Mr{\v{c}}un.
\newblock On isomorphisms of algebras of smooth functions.
\newblock {\em Proceedings of the American Mathematical Society},
  133(10):3109--3113, 2005.

\bibitem{shirota}
T.~Shirota.
\newblock A {G}eneralization of a {T}heorem of {I}. {K}aplansky.
\newblock {\em Osaka Mathematical Journal}, 4(2):121--132, 1952.

\bibitem{walker}
R.~C. Walker.
\newblock {\em The Stone-{\v{C}}ech compactification}.
\newblock Paper 566. {C}arnegie {M}ellon {U}niversity, 1972.

\bibitem{willard}
S.~Willard.
\newblock {\em General topology}.
\newblock Addison-Wesley Pub. Co., 1970.

\end{thebibliography}

\end{document}